\documentclass{amsart}
\usepackage{stmaryrd}
\usepackage{amssymb,epsfig,xspace}
\usepackage{ifpdf}
\usepackage{color}
\usepackage{bm}
\usepackage{booktabs}

\usepackage{tgpagella}

\usepackage{bbold}
\usepackage{mathrsfs}
\usepackage{esint}

\usepackage{enumerate}
\usepackage{soul}
\usepackage{commath}
\usepackage{mathtools}
\usepackage{eqparbox}

\newtheorem{theorem}{Theorem}[section]
\newtheorem{lemma}[theorem]{Lemma}
\newtheorem{proposition}[theorem]{Proposition}
\newtheorem{corollary}[theorem]{Corollary}

\theoremstyle{definition}

\theoremstyle{remark}
\newtheorem{remark}[theorem]{Remark}

\newtheorem{notations}[theorem]{Notations}

\numberwithin{equation}{section}

\newcommand{\1}{\mathbb{1}}

\newcommand{\R}{\mathbb{R}}

\newcommand{\cF}{\mathcal{F}}

\newcommand{\be}{\begin{equation}}
\newcommand{\ee}{\end{equation}}

\DeclareFontFamily{U}{mathx}{\hyphenchar\font45}
\DeclareFontShape{U}{mathx}{m}{n}{
      <5> <6> <7> <8> <9> <10>
      <10.95> <12> <14.4> <17.28> <20.74> <24.88>
      mathx10
      }{}
\DeclareSymbolFont{mathx}{U}{mathx}{m}{n}
\DeclareFontSubstitution{U}{mathx}{m}{n}
\DeclareMathAccent{\widecheck}{0}{mathx}{"71}
\DeclareMathAccent{\wideparen}{0}{mathx}{"75}

\newcommand{\tria}{{\mathcal T}}
\newcommand{\bbT}{\mathbb{T}}

\newcommand{\cP}{\mathcal{P}}

\newcommand{\identity}{\mathrm{Id}}

\newcommand{\V}{\mathscr{V}}
\newcommand{\U}{\mathscr{U}}
\newcommand{\B}{\mathscr{B}}
\newcommand{\Ss}{\mathscr{S}}
\newcommand{\W}{\mathscr{W}}
\newcommand{\Y}{\mathscr{Y}}
\newcommand{\ZZ}{\mathscr{Z}}

\DeclareMathOperator{\ran}{ran}

\DeclareMathOperator{\supp}{supp}

\DeclareMathOperator{\Span}{span}

\DeclareMathOperator{\diag}{diag}

\newcommand{\cL}{\mathcal L}
\newcommand{\Lis}{\cL\mathrm{is}}
\newcommand{\cLis}{\cL\mathrm{is}_c}
\newcommand{\lnrm}{\mathopen{| \! | \! |}}
\newcommand{\rnrm}{\mathclose{| \! | \! |}}

\newcommand{\trianu}{{\nu}}

\title{Uniform preconditioners for problems of positive order}

\date{\today}

\author{Rob Stevenson, Raymond van Veneti\"{e}}
\address{
Korteweg-de Vries Institute for Mathematics,
University of Amsterdam,
P.O. Box 94248,
1090 GE Amsterdam, The Netherlands}
\email{r.p.stevenson@uva.nl, r.vanvenetie@uva.nl}

\thanks{The first author has been supported by NSF Grant DMS 172029.
The second author has been supported by the Netherlands Organization for Scientific Research
(NWO) under contract.~no. 613.001.652}



\begin{document}
\begin{abstract}
Using the framework of operator or Cald\'{e}ron preconditioning, uniform preconditioners are constructed for elliptic operators of order $2s \in [0,2]$ discretized with continuous finite (or boundary) elements.
The cost of the preconditioner is the cost of the application an elliptic opposite order operator discretized with discontinuous or continuous finite elements on the same mesh, plus minor cost of linear complexity. Herewith the construction of a so-called dual mesh is avoided.
\end{abstract}

\maketitle

\section{Introduction}
This paper deals with the construction of \emph{uniform} preconditioners for operators of \emph{positive} order,
using the framework of `operator preconditioning' (\cite{138.26}).
It will build on our experiences with this approach for problems of negative order reported in~\cite{249.97}.

For some $d$-dimensional domain (or manifold) $\Omega$, a measurable, closed, possibly empty $\gamma \subset \partial\Omega$, and an $s \in [0,1]$,
we consider the Sobolev space
\[
\V:={[L_2(\Omega),H^1_{0,\gamma}(\Omega)]}_{s,2}.
\]
For $\V_\tria \subset \V$ a closed, e.g.~finite dimensional subspace, and $A_\tria \colon \V_\tria \to \V_\tria'$ some boundedly invertible linear operator,
we are interested in constructing a \emph{preconditioner} $G_\tria \colon \V_\tria' \to \V_\tria$.
More specifically, thinking of a a \emph{family} of spaces $\V_\tria$ and operators $A_\tria \colon \V_\tria \to \V_\tria'$, our aim is
to construct preconditioners $G_\tria$ such that $G_\tria A_\tria\colon \V_\tria \to \V_\tria$ is uniformly boundedly invertible.

It is well-known that such preconditioners of multi-level type are available. The advantage of operator preconditioning is, however, that it does not require a hierarchy of trial spaces.

In order to apply the operator preconditioning framework, one needs to  construct families of closed subspaces $\W_\tria \subset \W:=\V'$,
uniformly boundedly invertible $B_\tria \colon \W_\tria \to \W_\tria'$, and uniformly boundedly invertible  $D_\tria \colon \V_\tria \to \W_\tria'$.
Then the resulting preconditioners~$G_\tria$ are of the form
\[
G_\tria := D^{-1}_\tria B_\tria {(D'_\tria)}^{-1}.
\]

The canonical setting is that for $A\colon \V \rightarrow \V'$, i.e., an operator of order $2s$, and an \emph{opposite order} operator $B\colon \W \rightarrow \W'$,
both  boundedly invertible and coercive, it holds that
$(A_\tria u)(v) := (Au)(v)$ ($u,v \in \V_\tria$),
$(B_\tria u)(v) := (Bu)(v)$ ($u,v \in \W_\tria$), and $(D_\tria u)(v) := \langle u,v\rangle_{L_2(\Omega)}$ ($u \in \V_\tria, v \in \W_\tria$).
A typical example for $s=1/2$ is that $A$ is the \emph{Hypersingular Integral} operator, and $B$ is the  \emph{Weakly Singular Integral} operator.

A careful selection of $\W_\tria$ has to be made to ensure that $D_\tria \colon \V_\tria \to \W_\tria'$ is uniformly boundedly invertible.
A suitable family of $(\V_\tria, \W_\tria)$ pairs has been introduced in~\cite{249.16,35.8655}.
Here $\tria$ is a triangular partition of a \emph{two-dimensional} domain or manifold, $\V_\tria$ is the
space of continuous piecewise linears w.r.t.~$\tria$, and $\W_\tria$
is a subspace of the space of piecewise constants w.r.t.~a barycentric refinement of $\tria$,
constructed by subdividing each triangle into 6 subtriangles by connecting its vertices and midpoints with its barycenter.
It has been shown in~\cite{249.16, 138.275} that the preconditioner arising from these pairs $(\V_\tria, \W_\tria)$ is a uniform preconditioner
for families of partitions that satisfy a certain mildly-grading condition.

A problem with the constructions from~\cite{249.16,35.8655} appears when one considers the matrix representation $G_\tria$ in the standard bases, i.e.
$\bm{G}_\tria = \bm{D}_\tria^{-1}\bm{B}_\tria\bm{D}_\tria^{-\top}$.
Indeed, this matrix $\bm{D}_\tria$ is \emph{not} diagonal, and its inverse is densely populated so that it has to be
approximated.
Moreover, in order to get a uniform preconditioner, the accuracy with which $\bm{D}_\tria^{-1}$ has to be approximated increases with a decreasing mesh-size.
As a result, an application of $\bm{D}_\tria^{-1}$ cannot be expected to execute in linear time.

Another (practical) issue with these constructions is the need for the construction of the non-standard \emph{barycentrical} refinement of $\tria$.
This refinement increases the number of elements by a factor $6$, and therefore also increases the cost of evaluating $B_\tria : \W_\tria \to \W_\tria'$.

\subsection{Contributions}
With $\V_\tria$ being the space of continuous piecewise linears, the construction of $\W_\tria$ presented in this paper improves on the existing approach from~\cite{249.16,35.8655} concerning the following aspects:
\begin{itemize}
    \item The matrix representation $\bm{D}_\tria$ of $D_\tria$ will be diagonal, allowing one to (exactly) evaluate $\bm{D}_\tria^{-1}$ in linear time;
    \item The operator $G_\tria$ will be a uniformly well-conditioned preconditioner for families of uniformly shape regular partitions, without requiring
        a mildly-grading assumption on the partitions;
    \item By using a stable decomposition of $\W_\tria$ into a standard finite element space $\U_\tria$ w.r.t.~$\tria$ (either
    being the space of piecewise constants or $\U_\tria=\V_\tria$) and
    some bubble space, our $B_\tria$ will be the \emph{sum} of the corresponding Galerkin discretization operator of the opposite order operator  $B$, and an operator whose representation is a diagonal, with which the undesired barycentrical refinement is avoided;
    \item The construction of $\W_\tria$ applies in any space dimension, and extends to non piecewise planar
    manifolds.
\end{itemize}

We will extend the preconditioners to higher order finite element spaces by applying a subspace correction framework.

\subsection{Outline}
Sect.~\ref{sec:notations} recalls some notation that will be used throughout the article.
In Sect.~\ref{sec:theory} the general theory of operator preconditioning is summarized.
In Sect.~\ref{sec:preconditioning_cont_pw_lin}, the framework is specialized to operators of positive order discretized with \emph{continuous} piecewise linears.
Sect.~\ref{sec:decompose_b} give two constructions of $B_{\tria} \in \cLis(\W_\tria, \W_\tria')$ that avoid any refinement of the partition $\tria$ that underlies the trial space $\V_\tria$.
In Sect.~\ref{sec:extensions} the preconditioners are generalized to higher order finite element spaces, and to spaces defined on manifolds.
Finally, in Sect.~\ref{sec:numerical_results} we report some numerical results obtained with the new preconditioners.

\section{Notations}\label{sec:notations}
\begin{notations}
In this work, by $\lambda \lesssim \mu$ we will mean that $\lambda$ can be bounded by a multiple of $\mu$, independently of parameters which $\lambda$ and $\mu$ may depend on, with the sole exception of the space dimension $d$, or in the manifold case, on the parametrization of the manifold that is used to define the finite element spaces on it. Obviously, $\lambda \gtrsim \mu$ is defined as $\mu \lesssim \lambda$, and $\lambda\eqsim \mu$ as $\lambda\lesssim \mu$ and $\lambda \gtrsim \mu$.
\end{notations}

\begin{notations}
For normed linear spaces $\Y$ and $\ZZ$, in this paper for convenience over $\R$, $\cL(\Y,\ZZ)$ will denote the space of bounded linear mappings $\Y \rightarrow \ZZ$ endowed with the operator norm $\|\cdot\|_{\cL(\Y,\ZZ)}$. The subset of invertible operators in $\cL(\Y,\ZZ)$  with inverses in $\cL(\ZZ,\Y)$
will be denoted as $\Lis(\Y,\ZZ)$.
The \emph{condition number} of a $C \in \Lis(\Y,\ZZ)$ is defined as $\kappa_{\Y,\ZZ}(C):=\|C\|_{\cL(\Y,\ZZ)}\|C^{-1}\|_{\cL(\ZZ,\Y)}$.

For $\Y$ a reflexive Banach space and $C \in \cL(\Y,\Y')$ being \emph{coercive}, i.e.,
\[
\inf_{0 \neq y \in \Y} \frac{(Cy)(y)}{\|y\|^2_\Y} >0,
\]
both $C$ and $\Re(C)\!:= \!\frac{1}{2}(C+C')$ are in $\Lis(\Y,\Y')$ with
\begin{align*}
\|\Re(C)\|_{\cL(\Y,\Y')} &\leq \|C\|_{\cL(\Y,\Y')},\\
\|C^{-1}\|_{\cL(\Y',\Y)} & \leq \|\Re(C)^{-1}\|_{\cL(\Y',\Y)}=\Big(\inf_{0 \neq y \in \Y} \frac{(Cy)(y)}{\|y\|^2_\Y}\Big)^{-1}.
\end{align*}
%and so $\max(\kappa_{\Y,\Y'}(C),\kappa_{\Y,\Y'}(\Re(C)))\leq \|C\|_{\cL(\Y,\Y')} \|\Re(C)^{-1}\|_{\cL(\Y',\Y)}$.
The set of coercive $C \in \Lis(\Y,\Y')$ is denoted as $\cLis(\Y,\Y')$.
If $C   \in \cLis(\Y,\Y')$, then $C^{-1} \in \cLis(\Y',\Y)$ and $\|\Re(C^{-1})^{-1}\|_{\cL(\Y,\Y')} \leq \|C\|_{\cL(\Y,\Y')}^2 \|\Re(C)^{-1}\|_{\cL(\Y',\Y)}$.

Given a family of operators $C_i \in \Lis(\Y_i,\ZZ_i)$ ($\cLis(\Y_i,\ZZ_i)$), we will write $C_i \in \Lis(\Y_i,\ZZ_i)$ ($\cLis(\Y_i,\ZZ_i)$) uniformly in $i$, or simply `uniform', when
    \[
    \sup_{i} \max(\|C_i\|_{\cL(\Y_i,\ZZ_i)},\|C_i^{-1}\|_{\cL(\ZZ_i,\Y_i)})<\infty,
     \]
or
    \[
    \sup_{i} \max(\|C_i\|_{\cL(\Y_i,\ZZ_i)},\|\Re(C_i)^{-1}\|_{\cL(\ZZ_i,\Y_i)})<\infty.
     \]
\end{notations}
 \begin{notations}
Given a finite collection $\Upsilon=\{\upsilon\}$ in a linear space, we set the \emph{synthesis operator}
\[
\cF_\Upsilon:\R^{\# \Upsilon} \rightarrow \Span\Upsilon\colon {\bf c} \mapsto {\bf c}^\top \Upsilon:=\sum_{\upsilon \in \Upsilon} c_\upsilon \upsilon.
\]
Equipping $\R^{\# \Upsilon}$ with the Euclidean scalar product $\langle \cdot, \cdot \rangle$, and identifying $(\R^{\# \Upsilon})'$ with~$\R^{\# \Upsilon}$ using the corresponding Riesz map, we infer that the adjoint of $\cF_\Upsilon$, known as the \emph{analysis operator}, satisfies
\[
\cF'_\Upsilon:(\Span\Upsilon)' \rightarrow \R^{\# \Upsilon}:f \mapsto f(\Upsilon):=[f(\upsilon)]_{\upsilon \in \Upsilon}.
\]
A collection $\Upsilon$ is a \emph{basis} for its span when $\cF_\Upsilon\in \Lis(\R^{\# \Upsilon},\Span\Upsilon)$ (and so $\cF'_\Upsilon \in \Lis((\Span\Upsilon)' ,\R^{\# \Upsilon})$.)

Two countable collections $\Upsilon=(\upsilon_i)_i$ and $\tilde{\Upsilon}=(\tilde{\upsilon}_i)_i$ in a Hilbert space will be called \emph{biorthogonal} when $\langle \Upsilon, \tilde{\Upsilon}\rangle=[\langle \upsilon_j, \tilde{\upsilon}_i\rangle]_{i j}$ is an \emph{invertible diagonal} matrix, and \emph{biorthonormal} when it is the \emph{identity} matrix.
\end{notations}

\section{Operator preconditioning}\label{sec:theory}
We shortly recap the idea of opposite order preconditioning,
which is based on the following result, see~\cite[Sect.~2]{138.26}.
\begin{proposition}\label{prop:opposite_preconditioner}
Let $\V, \W$ be reflexive Banach spaces.

If $B \in \Lis(\W,\W')$ and $D  \in \Lis(\V,\W')$, then
\[
G:=D^{-1} B (D')^{-1} \in \Lis(\V',\V),
\]
and
\begin{align*}
\|G\|_{\cL(\V',\V)} &\leq \|D^{-1}\|_{\cL(\W',\V)}^2 \|B\|_{\cL(\W,\W')},\\
\|G^{-1}\|_{\cL(\V,\V')} & \leq \|D\|_{\cL(\V,\W')}^2 \|B^{-1}\|_{\cL(\W',\W)}.
\end{align*}
If additionally $B \in \cLis(\W,\W')$, then $G \in \cLis(\V',\V)$, and
\[
\|\Re(G)^{-1}\|_{\cL(\V,\V')} \leq \|D\|_{\cL(\V,\W')}^2 \|\Re(B)^{-1}\|_{\cL(\W',\W)}.
\]
\end{proposition}
Let be given families of finite dimensional spaces $\V_\tria$ for $\tria \in \bbT$, and operators $A_\tria \in \Lis(\V_{\tria},\V_{\tria}')$ uniformly in $\tria \in \bbT$.
Then in light of Proposition~\ref{prop:opposite_preconditioner} we will seek preconditioners for $A_\tria$ of the form
\begin{equation}\label{eq:preconditioner}
    G_\tria=D_\tria^{-1} B_\tria (D_\tria')^{-1},
\end{equation}
where $B_\tria \in \Lis(\W_\tria,\W_\tria')$ and $D_\tria  \in \Lis(\V_\tria,\W_\tria')$ (both uniformly in $\tria \in \bbT$), and
\begin{equation}\label{eq:equaldims}
\dim \W_\tria=\dim \V_\tria.
\end{equation}

A typical situation is that for some reflexive Banach space $\V$ and $A \in \cLis(\V,\V')$,  it holds that
$\V_\tria \subset \V$ (thus \emph{equipped with~$\|\cdot\|_\V$}) and $(A_\tria u)(v):=(A u)(v)$ ($u,v \in \V_\tria$), so that indeed
$A_\tria \in \cLis(\V_{\tria},\V_{\tria}')$ uniformly in $\tria \in \bbT$.
Then for a suitable reflexive Banach space $\W$, an operator $B \in \cLis(\W, \W')$, and a subspace $\W_\tria \subset \W$ (thus \emph{equipped with~$\|\cdot\|_\W$}), one
can take $(B_\tria w)(z):=(B w)(z)$ ($w,z \in \W_\tria$), giving $B_\tria \in \cLis(\W_\tria, \W_\tria')$ uniformly.
A possible construction of $D_\tria \in \Lis(\V_\tria, \W_\tria')$ uniformly is discussed in the next proposition.

\begin{proposition}[Fortin projector (\cite{75.22})]\label{prop:fortin}
For some $D  \in \Lis(\V,\W')$, let $D_\tria \in \cL(\V_\tria,\W_\tria')$ be defined by
$(D_\tria v)(w):=(D v)(w)$. Then
\begin{align} \nonumber
\|D_\tria\|_{\cL(\V_\tria,\W_\tria')} &\leq  \|D\|_{\cL(\V,\W')}.
\intertext{Assuming \eqref{eq:equaldims}, additionally one has $D_\tria \in \Lis(\V_\tria,\W_\tria')$ \emph{if}, and for $\W$ being a Hilbert space, \emph{only if} there exists a projector $P_\tria \in \cL(\W,\W)$ onto $\W_\tria$ with \mbox{$(D \V_\tria )((\identity-P_\tria)\W)=0$,} in which case}
\label{bound}
\|D_\tria^{-1}\|_{\cL(\W_\tria',\V_\tria)} &\leq \|P_\tria\|_{\cL(\W,\W)} \|D^{-1}\|_{\cL(\W',\V)}.
\end{align}
\end{proposition}

In our applications, the choices for $\W$ and $D$ will be obvious, and the key ingredient for the construction of a uniform preconditioner $G_\tria$ will be the selection of $\W_\tria$ that allows for a uniformly bounded Fortin projector $P_\tria$.

\subsection{Implementation}\label{sec:theory_impl}
Let $\Phi_\tria=(\phi_i)_i$ and $\Psi_\tria=(\psi_i)_i$ be bases for $\V_\tria$ and $\W_\tria$, respectively.
Then in coordinates the preconditioned system reads as
\[
\cF_{\Phi_\tria}^{-1} G_\tria A_\tria \cF_{\Phi_\tria}= \bm{G}_\tria \bm{A}_\tria:=\bm{D}_\tria^{-1} \bm{B}_\tria \bm{D}_\tria^{-\top} \bm{A}_\tria,
\]
where
\[
\bm{A}_\tria:=\cF'_{\Phi_\tria} A_\tria \cF_{\Phi_\tria},\quad
\bm{B}_\tria:=\cF'_{\Psi_\tria} B_\tria \cF_{\Psi_\tria},\quad
\bm{D}_\tria:=\cF'_{\Psi_\tria} D_\tria \cF_{\Phi_\tria}.
\]

By identifying a map in $\cL(\R^{\# \Phi_\tria},\R^{\# \Phi_\tria})$ with a ${\# \Phi_\tria} \times {\# \Phi_\tria}$ matrix by equipping~$\R^{\# \Phi_\tria}$ with the canonical basis $(\bm{e}_i)_i$ one has,
\[
(\bm{A}_\tria)_{ij}=\langle \cF'_{\Phi_\tria} A_\tria \cF_{\Phi_\tria} \bm{e}_{j}, \bm{e}_{i}\rangle=(A_\tria \cF_{\Phi_\tria}  \bm{e}_{j})(\cF_{\Phi_\tria}  \bm{e}_{i})=(A_\tria  \phi_{j})(\phi_{i}),
\]
and similarly,
\[
(\bm{B}_\tria)_{ij}=(B_\tria\psi_{j})(\psi_{i}), \quad (\bm{D}_\tria)_{ij}=(D_\tria \phi_{j})(\psi_{i}).
\]
Preferably $\bm{D}_\tria$ is such that its inverse can be applied in linear complexity, as is the case when $\bm{D}_\tria$ is \emph{diagonal}.
A goal of this paper is to construct such a diagonal $\bm{D}_\tria$.
\begin{remark}
Using $\sigma(\cdot)$ and $\rho(\cdot)$ to denote the spectrum and spectral radius of an operator,
clearly $\sigma(\bm{G}_\tria \bm{A}_\tria)=\sigma(G_\tria   A_\tria)$.  So for the spectral condition number we have
\[
\kappa_S(\bm{G}_\tria \bm{A}_\tria):=\rho(\bm{G}_\tria \bm{A}_\tria)\rho((\bm{G}_\tria \bm{A}_\tria)^{-1}) \leq \kappa_{\V_\tria,\V_\tria}(G_\tria A_\tria),
\]
which thus holds true \emph{independently} of the choice of the basis $\Phi_\tria$ for $\V_\tria$.
Furthermore, in view of an application of Conjugate Gradients, if $A_\tria$ and $B_\tria$ are coercive and \emph{self-adjoint}, then $\bm{A}_\tria$ and $\bm{G}_\tria$ are positive definite and symmetric.
Equipping $\R^{\dim \V_\tria}$ with $\lnrm\cdot\rnrm:=\|(\bm{G}_\tria)^{-\frac{1}{2}}\cdot\|$ or $\lnrm\cdot\rnrm:=\|(\bm{A}_\tria)^{\frac{1}{2}}\cdot\|$, in that case we have
\[
\kappa_{(\R^{\dim \V_\tria},\lnrm\cdot\rnrm),(\R^{\dim \V_\tria},\lnrm\cdot\rnrm)}(\bm{G}_\tria \bm{A}_\tria)=\kappa_S(\bm{G}_\tria \bm{A}_\tria)
.
\]
\end{remark}

\section{Application to operators discretized with continuous piecewise linears}\label{sec:preconditioning_cont_pw_lin}
For a bounded polytopal domain $\Omega \subset \R^d$, a measurable, closed, possibly empty $\gamma \subset \partial\Omega$, and an $s \in [0,1]$, we take
\[
    \V:={[L_2(\Omega),H^1_{0,\gamma}(\Omega)]}_{s,2},\quad \W:=\V',
\]
which forms the Gelfand triple $\V \hookrightarrow L_2(\Omega) \simeq L_2(\Omega)' \hookrightarrow \W$. We define the operator $D \in \Lis(\V, \W')$ as the unique extension
to $\V \times \W$ of the duality pairing
\[
 (Dv)(w) := \langle v,w \rangle_{L_2(\Omega)},
\] which satisfies $\|D\|_{\cL(\V,\W')}=\|D^{-1}\|_{\cL(\W',\V)}=1$.

Let $(\tria)_{\tria \in \bbT}$ be a family of conforming partitions of $\Omega$ into closed uniformly shape regular $d$-simplices.
Thanks to the conformity and the uniform shape regularity, for $d>1$ we know that neighbouring $T,T' \in \tria$, i.e.~$T \cap T' \neq \emptyset$, have uniformly comparable sizes.
 For $d=1$, we impose this uniform `\emph{$K$-mesh property}' explicitly.

For some $\tria \in \bbT$, denote $N^0_\tria$ as the subset of vertices that are not on~$\gamma$, where we assume that $\gamma$ is the (possibly empty) union of $(d-1)$-faces of $T \in \tria$.
For $T \in \tria$, write $N_T$ for the set of its vertices, set $N^0_T := N^0_\tria \cap N_T$, $h_T := |T|^{1/d}$, and the piecewise constant function $h_\tria$ by $h_\tria|_T=h_T$ ($T \in \tria$).
For any vertex $\nu \in N^0_\tria$, define the patch $\omega_{\tria, \nu} := \bigcup_{\{T \in \tria : \nu \in T\}} T$ and the local mesh size $h_{\tria, \nu}:= |\omega_{\tria, \nu}|^{1/d}$.
We omit notational dependence on $\tria$ if it is clear from the context, and simply write $\omega_\trianu$ and $h_\trianu$.

Let the discretization space $\V_\tria$ be the space of continuous piecewise linears, zero on $\gamma$,
\[
  \V_\tria = \Ss^{0,1}_\tria:=\{u \in H^1_{0,\gamma}(\Omega)\colon u|_T \in \cP_1 \,(T \in \tria)\} \subset \V,
\]
equipped with the nodal bases
\[
\Phi_\tria = \{\phi_{\trianu}\colon \nu \in N^0_\tria\}
\]
defined by $\phi_{\trianu}(\nu') := \delta_{\nu\nu'}$ ($\nu, \nu' \in N^0_\tria$).
For future reference, define the space of discontinuous piecewise constants by
\[
\Ss_\tria^{-1,0} :=\{u \in L_2(\Omega)\colon u|_T \in \cP_0 \,(T \in \tria)\} \subset \W,
\]
equipped with the basis
\begin{equation}\label{sigma}
\Sigma_\tria:=\{\1_T\colon T \in \tria\},
\end{equation}
where $\1_K$ is defined by, for any $K \subseteq \Omega$,
    $\1_K := 1$ on $K$, and $\1_K := 0$ elsewhere.

\subsection{The subspace $\W_\tria$}\label{sec:assumptions_w}
We will construct the preconditioning space $\W_\tria$ as
\[
    \W_\tria := \Span \Psi_\tria \subset \W, \,\, \text{with}\,\, \dim \W_\tria = \dim \V_\tria
\]
for a collection $\Psi_\tria \subset L_2(\Omega)$ that is biorthogonal to $\Phi_\tria$,
and for which the biorthogonal projector $P_\tria \in \cL(\W,\W)$ onto $\W_\tria$ is uniformly bounded.
We require the collection $\Psi_\tria := \{ \psi_\trianu \in \W \colon \nu \in N_\tria^0\}$ to satisfy
\begin{equation}
    \label{eq:assumption_psi}
    \begin{aligned}
       \big| \langle \phi_\trianu, \psi_{\trianu'} \rangle_{L_2(\Omega)} \big| &\eqsim \delta_{\nu \nu'} \| \phi_\trianu \|_{L_2(\Omega)} \| \psi_{\trianu'} \|_{L_2(\Omega)} \quad (\nu, \nu' \in N_\tria^0), \\
        \supp \psi_{\trianu} &\subseteq \omega_{\trianu} \quad (\nu \in N_\tria^0).
    \end{aligned}
\end{equation}
Existence of such collections will be shown later in Sect.~\ref{sec:decompose_b}.
%It is desirable, for practical purposes, that the functions $\psi_{\trianu}$  are in some standard finite element space, e.g.~piecewise constants on some (refined) partition.

\subsection{Bounded Fortin projector}
From the assumptions \eqref{eq:assumption_psi} it follows that the biorthogonal Fortin projector $P_\tria \colon H_{0,\gamma}^1(\Omega)' \to L_2(\Omega)$ onto~$\W_\tria$ with $\ran (I-P_\tria) = \V_\tria^{\perp_{L_2(\Omega)}}$ exists, and is given by
\[
    P_\tria u = \sum_{ \nu \in N^0_\tria} \frac{ \langle u, \phi_{\trianu} \rangle_{L_2(\Omega)}} { \langle \phi_{\trianu}, \psi_{\trianu} \rangle_{L_2(\Omega)}} \psi_{\trianu}.
\]
Uniform boundedness of $\|P_\tria\|_{\cL(\W,\W)}$ follows from uniform boundedness
of its adjoint $P_\tria'$, which can be shown similarly as in~\cite[Thm.~3.2]{249.97}\footnote{Note that the roles of $\V$ and $\W$ are interchanged compared to~\cite{249.97}.}:
\begin{theorem}\label{thm:boundedbiorth}
    It holds that $\sup_{\tria \in \bbT} \| P_\tria\|_{\cL(\W, \W)} = \sup_{\tria \in \bbT} \| P_\tria'\|_{\cL(\V,\V)} < \infty$.
\end{theorem}
\begin{proof}
    Let $\tria \in \bbT$. Define $\omega^{(0)}_T:=T$ for $T \in \tria$, and for $i=1,\ldots$, denote $\omega^{(i)}_T:=\bigcup_{\{T' \in \tria \colon T' \cap \omega_T^{(i-1)} \ne \emptyset\}} T'$. The adjoint $P_\tria'\colon L_2(\Omega) \to H_{0,\gamma}^1(\Omega)$ onto $\V_\tria$ is given by
    \[
        P_\tria' u = \sum_{\nu \in N^0_\tria} \frac{ \langle u, \psi_{\trianu} \rangle_{L_2(\Omega)}}{ \langle \phi_{\trianu}, \psi_{\trianu} \rangle_{L_2(\Omega)}} \phi_{\trianu}.
    \]
    Properties of the nodal basis functions, $\|\phi_\trianu\|^2_{L_2(\Omega)} \eqsim h_\trianu^{d}$ and $\|\phi_\trianu\|^2_{H^1(\Omega)} \lesssim h_\trianu^{d - 2}$, in combination with~\eqref{eq:assumption_psi}, can be used to show that, for $T \in \tria$ and $k \in \{0,1\}$,
    \begin{equation}
        \begin{aligned}
         \label{eq:proj_l2_bound}
        \| P_\tria' u\|_{H^k(T)} &\leq \sum_{\nu \in N^0_T} \| \phi_{\trianu} \|_{H^k(T)}
        \frac{\|u\|_{L_2(\supp \psi_\trianu)} \| \psi_{\trianu}\|_{L_2(\Omega)}}{| \langle \phi_{\trianu}, \psi_{\trianu}\rangle_{L_2(\Omega)}| }\\
        &\lesssim \sum_{\nu \in N^0_T} h_{\trianu}^{-k} \|u\|_{L_2(\supp \psi_\trianu)} \lesssim h_T^{-k} \|u\|_{L_2(\omega^{(1)}_T)},
    \end{aligned}
    \end{equation}
    from which we may directly conclude that
    \[
        \sup_{\tria \in \bbT} \| P_\tria' \|_{\cL(L_2(\Omega), L_2(\Omega))} < \infty.
    \]

    For proving boundedness in $H^1_{0,\gamma}(\Omega)$,
    we consider the Scott-Zhang (\cite{247.2}) interpolator $\Pi_\tria: H_{0,\gamma}^1(\Omega) \to \V_\tria$.
    From~\eqref{eq:proj_l2_bound} and properties of the $\Pi_\tria$~\cite[(3.8) and (4.3)]{247.2}, we deduce that
    \begin{align*}
        \| P_\tria' u\|_{H^1(T)} &= \| \Pi_\tria u  + P'_\tria(\identity - \Pi_\tria)u\|_{H^1(T)} \\
                                  &\lesssim \| u\|_{H^1(\omega_\tria^{(1)}(T))} +h_T^{-1} \| (\identity - \Pi_\tria)u\|_{L_2(\omega_\tria^{(1)}(T))} \\
                                  &\lesssim \| u\|_{H^1(\omega_\tria^{(2)}(T))},
    \end{align*}
    and consequently
    \[
        \sup_{\tria \in \bbT} \| P_\tria'\|_{\cL(H_{0,\gamma}^1(\Omega), H_{0,\gamma}^1(\Omega))} < \infty.
    \]
    An application of the Riesz-Thorin interpolation theorem yields the result.
\end{proof}
The basis $\Psi_\tria$ has the crucial benefit that the matrix representation of $D_\tria$, i.e.
\[
\bm{D}_\tria = \langle \Phi_\tria, \Psi_\tria \rangle_{L_2(\Omega)},
\]
is \emph{diagonal}, and thus easily invertible, cf.\ Sect.~\ref{sec:theory_impl}.

Combining the theorem with Proposition~\ref{prop:fortin} gives the following corollary (without requiring additional
assumptions on the family of partitions $\bbT$).
\begin{corollary}\label{cor:uniform_preconditioner}
    Suppose we have $B_\tria \in \cLis(\W_\tria,\W_\tria')$ uniformly.
    With $D_\tria\colon \V_\tria \to \W_\tria$ defined by $(D_\tria v)(w) := \langle v, w \rangle_{L_2(\Omega)}$, we find
    that $G_\tria = D^{-1}_\tria B_\tria {(D'_\tria)}^{-1} \in \cLis(\V_\tria', \V_\tria)$ is a uniform preconditioner of $A_\tria \in \cLis(\V_\tria, \V_\tria')$.
\end{corollary}

Given some $B \in \cLis(\W, \W')$, a possible choice for $B_\tria  \in \cLis(\W_\tria, \W_\tria')$ uniformly in $\tria \in \bbT$,
is $(B_\tria  u)(v):=(B u)(v)$ ($u,v \in \W_\tria$). For $d \in \{2,3\}$ and $\W'=\V=H^{\frac{1}{2}}(\Omega)$, a suitable $B$ is given by the Weakly Singular Integral operator, whereas for
$\W'=\V=H^{\frac{1}{2}}_{00}(\Omega):=[L_2(\Omega),H^1_0(\Omega)]_{\frac{1}{2},2}$, the
recently in~\cite{138.28}  introduced \emph{Modified} Weakly Singular Integral operator can be applied. Similar comments apply to screens.

\section{Construction of $B_\tria \in \cLis(\W_\tria, \W_\tria')$}\label{sec:decompose_b}
We expect it to be impossible to construct a basis $\Psi_\tria$ in the
(standard) spaces~$\Ss_{\tria}^{-1,0}$ or $\Ss_{\tria}^{0,1}$ that is \emph{local} and \emph{biorthogonal} to $\Phi_\tria$ as required in assumption~\eqref{eq:assumption_psi}.
One remedy is to construct $\Psi_\tria$ in a (finite element) space w.r.t.~a refined partition $\tria_*\succ \tria$.
However, this implies that some opposite order operator $B \in \cLis(\W, \W')$ has to be discretized on a space w.r.t.~the \emph{refined} partition~$\tria_*$.
This increases the cost of the preconditioner, and moreover, increases implementational complexity as one has to actually construct this refined partition.

To circumvent (explicit) dependence on the refined partition $\tria_*$, we shall apply the idea described in~\cite[Sec.~3.3]{249.97}.
 That is, we will construct an operator $B_\tria \in \cLis(\W_\tria, \W_\tria')$ by decomposing the space $\W_\tria$ into a
a standard finite element space~$\U_\tria$, either $\Ss^{-1,0}_\tria$ (in Sect.~\ref{sec:decompose_pw_const}) or $\Ss^{0,1}_\tria$ (in Sect.~\ref{sec:decompose_cont_pw_lin}), and some bubble space $\B_\tria$.
On $\U_\tria$ we will apply the Galerkin discretization operator of the opposite order operator $B$,
whereas on the bubble space $\B_\tria$ a diagonal scaling will suffice.

In the first subsection we present this construction of $B_\tria$ for some abstract $\W_\tria$. In the subsequent subsections, we will present two viable options for $\W_\tria$, leading
to two different preconditioners.

\subsection{Stable decomposition}\label{stabledecomp}
The role of the space `$\Y$' is the next abstract proposition is going to be played by $\W_\tria$.
\begin{proposition}\label{proppie} Let $\ZZ$ be an inner product space, $Q \in \cL(\ZZ,\ZZ)$ a projector, and with $\U:=\ran Q$, let $\B:=\ran(\identity -Q)$, $B^\U \in \cLis(\U,\U')$, and $B^\B \in \cLis(\B,\B')$.
Then for any subspace $\Y \subset \ZZ$,
\[
(B y)(\tilde y):=(B^\U Q y)(Q \tilde y) + (B^\B( \identity - Q) y)((\identity - Q) \tilde y) \quad (y,\tilde y \in \Y),
\]
is bounded and coercive --- $B \in \cLis(\Y, \Y')$ --- with
 \begin{align*}
        \|B\|_{\cL(\Y, \Y')} &\leq \\
        & \hspace*{-4em}\Big(\|Q\|^2+\sqrt{\|Q\|^4-\|Q\|^2}\Big)
        \max(\|B^\U\|_{\cL(\U, \U')}, \|B^\B\|_{\cL(\B, \B')}),\\
        \|\Re(B)^{-1}\|_{\cL(\Y', \Y)} &\leq \\
        &\hspace*{-4em}\Big(1+\sqrt{1-\|Q\|^{-2}}\Big)\max(\|\Re(B^\U)^{-1}\|_{\cL(\U', \U)}, \|\Re(B^\B)^{-1}\|_{\cL(\B', \B)}),
    \end{align*}
where $\|Q\|:=\|Q\|_{\cL(\ZZ,\ZZ)}$.
\end{proposition}

\begin{proof}
    Let $y, \tilde y \in \Y$. Write $u = Q y$, $b = (\identity - Q)y$, and similarly $\tilde u = Q \tilde y$, $\tilde b = (\identity - Q)\tilde y$.
    We have
    \begin{align*}
    |(B(y))(\tilde y)| \leq &\max\left(\|B^\U\|_{\cL(\U, \U')}, \|B^\B\|_{\cL(\B, \B')}\right) \cdot \left(\|u\|_\ZZ \|\tilde u\|_\ZZ + \|b\|_\ZZ \| \tilde b\|_\ZZ\right)\\
                            \leq&\max(\cdots) \sqrt{\|u\|_\ZZ^2 + \|b\|_\ZZ^2}\cdot\sqrt{\|\tilde u\|_\ZZ^2 + \|\tilde b\|_\ZZ^2},
    \end{align*}
    and
    \[
    |(B(y))(y)| \geq \min\left(\|\Re(B^\U)^{-1}\|^{-1}_{\cL(\U', \U)}, \| \Re(B^\B)^{-1}\|_{\cL(\B', \B)}^{-1}\right) \cdot (\|u\|_\ZZ^2 + \|b\|_\ZZ^2).
    \]
    With $\gamma:=\sup_{0 \neq (u,b)\in \U \times \B} \frac{|\langle u,b\rangle_\ZZ|}{\|u\|_\ZZ\|b\|_\ZZ}$, for $0 \neq (u,b)\in \U \times \B$ we have $\frac{\|u+b\|_\ZZ^2}{\|u\|_\ZZ^2 +\|b\|_\ZZ^2} \in [1-\gamma,1+\gamma]$.
    Using that $\sqrt{\frac{1}{1-\gamma^2}}=\|Q\|$ (see e.g. \cite[(5.5), (5.7), (6.2)]{257}), the proof is easily completed.
\end{proof}

\begin{remark}
    For a quantitatively weaker result as Proposition~\ref{proppie} to hold it is actually sufficient when $Q$ is only defined on $\Y$,
    and neither is it needed that it is a projector.
    Under these relaxed conditions, obvious estimates show bounds as in Proposition~\ref{proppie} with the factors $\|Q\|^2+\sqrt{\|Q\|^4-\|Q\|^2}$ and $1+\sqrt{1-\|Q\|^{-2}}$ reading as $\|Q|_\Y\|^2+(1+\|Q|_\Y\|)^2$ and $2$, respectively. Both original factors are equal to $1$ when $Q$ is an orthogonal projector.
\end{remark}

We are going to apply this abstract proposition with `$\Y$'$=\W_\tria$,
`$\U$'$=\U_\tria$ being a standard finite element space,
`$\B$'$=\B_\tria$ being a suitably constructed `bubble space',
and `$\ZZ$'$=\ZZ_\tria:=\U_\tria+\B_\tria$, all equipped with the norm on $\W$.
The resulting `$B$' will be the $B_\tria \in \cLis(\W_\tria,\W_\tria')$ we are seeking.

In order to apply above proposition, what is left is the construction of a (uniformly) bounded projector defined on $\ZZ_\tria$.
Furthermore, to allow for a simple preconditioner on $\B_\tria$ we would like to find a setting in which on this bubble space the $\W$-norm is equivalent to a weighted $L_2$-norm.
Both issues will be dealt with in the next two lemmas. The operator $Q_\tria|_{\ZZ_\tria}$ in the first lemma will play the role of `$Q$' in Proposition~\ref{proppie}.

\begin{lemma}\label{lem:boundedW}
    Let $Q_\tria \in \cL(L_2(\Omega),H^1_{0,\gamma}(\Omega)')$ be a projector, $\U_\tria \subseteq \ran Q_\tria$ and $\B_\tria \subseteq \ran(\identity-Q_\tria)$ be subspaces of $L_2(\Omega)$,  and with $\ZZ_\tria:=\U_\tria+\B_\tria$, let
     \begin{enumerate}
    \item $\| h_\tria^{-1}(\identity - Q_\tria')\|_{\cL(H^1_{0,\gamma}(\Omega),L_2(\Omega))} \lesssim 1$, \hfill(approximation property)\label{enum:approx}
    \item $\sup_{\tria \in \bbT} \| Q_\tria |_{\ZZ_\tria} \|_{\cL((\ZZ_\tria,\|\cdot\|_{L_2(\Omega)}), L_2(\Omega))} \lesssim 1$, \hfill(boundedness in $L_2(\Omega)$)\label{enum:boundedl2}
  \item $\| h_\tria \cdot \|_{L_2(\Omega)} \lesssim \| \cdot \|_{H_{0,\gamma}^1(\Omega)'}$ on $\ZZ_\tria$. \hfill(inverse inequality)\label{enum:invinequality}
   \end{enumerate}
    Then $Q_\tria|_{\ZZ_\tria} \colon \ZZ_\tria \rightarrow \ZZ_\tria$ is a projector, $\ran Q_\tria|_{\ZZ_\tria}=\U_\tria$, $\ran (\identity-Q_\tria|_{\ZZ_\tria})=\B_\tria$, and
    \begin{enumerate}[(i)]
        \item $\sup_{\tria \in \bbT} \| Q_\tria |_{\ZZ_\tria} \|_{\cL((\ZZ_\tria,\|\cdot\|_\W), \W)} < \infty$,
       \item $\| \cdot \|_{\W} \eqsim \| h_\tria^s \cdot \|_{L_2(\Omega)}$ on $\B_\tria$.
    \end{enumerate}
\end{lemma}
\begin{proof} The first three statements are easily verified.
   From~\eqref{enum:approx} it follows that for $u \in H_{0,\gamma}^1(\Omega)'$:
    \begin{align*}
        \| (\identity - Q_\tria)u\|_{H^1_{0,\gamma}(\Omega)'} &= \sup_{v \in H^1_{0,\gamma}(\Omega)} \frac{ \langle u, (\identity - Q_\tria')v\rangle_{L_2(\Omega)}}{\| v\|_{H_{0,\gamma}^1(\Omega)}} \\
                                                                     &\leq \sup_{v \in H^1_{0,\gamma}(\Omega)} \frac{ \|h_\tria u\|_{L_2(\Omega)} \|h_\tria^{-1}(\identity - Q_\tria') v\|_{L_2(\Omega)}}{ \| v\|_{H_{0,\gamma}^1(\Omega)}}\\
                                                                     &\lesssim \|h_\tria u\|_{L_2(\Omega)}.
    \end{align*}

    Together with the inverse inequality on $\ZZ_\tria$, this  gives (uniform) boundedness of $\|(\identity - Q_\tria)|_{\ZZ_\tria} \|_{\cL((\ZZ_\tria,\|\cdot\|_{H^1_{0,\gamma}(\Omega)'}), H^1_{0,\gamma}(\Omega)')}$
    and thus of
    $\|{Q_\tria}|_{\ZZ_\tria} \|_{\cL((\ZZ_\tria,\|\cdot\|_{H^1_{0,\gamma}(\Omega)'}), H^1_{0,\gamma}(\Omega)')}$.  The first result then follows from~\eqref{enum:boundedl2} and an interpolation argument.

    By the inverse inequality on $\B_\tria$ and the previously derived inequality, we have for $b_\tria \in \B_\tria  \subseteq \ran(\identity-Q_\tria)$  that
    \[
  	\|b_\tria\|_{H_{0,\gamma}^1(\Omega)'} = \| (\identity - Q_\tria)b_\tria\|_{H_{0,\gamma}^1(\Omega)'} \lesssim \|h_\tria b_\tria\|_{L_2(\Omega)} \lesssim \|b_\tria\|_{H_{0,\gamma}^1(\Omega)'}.
    \]
   Another interpolation argument yields the second result.
\end{proof}
\begin{lemma}\label{lem:diagonal_b}
    Suppose that $\| \cdot \|_\W \eqsim \| h_\tria^s \cdot \|_{L_2(\Omega)}$ holds on $\B_\tria$,
    and that $\Theta_\tria$ is a \emph{uniformly $\|h_\tria^s \cdot\|_{L_2(\Omega)}$-stable basis for $\B_\tria$}, i.e.
\[
\B_\tria = \Span \Theta_\tria \quad \text{and} \quad \big\|h_\tria^s \sum\nolimits_{\theta \in \Theta_\tria} c_\theta \theta\big\|^2_{L_2(\Omega)} \eqsim \sum\nolimits_{\theta \in \Theta_\tria} |c_\theta|^2 \|h_\tria^s \theta\|^2_{L_2(\Omega)},
\]
then, for any $\beta_1 > 0$, an operator $B_\tria^\B \in \cLis(\B_\tria, \B_\tria')$ is given by
\begin{equation}\label{eq:bubble_b}
    \big(B_\tria^\B \sum_{\theta \in \Theta_\tria} c_{\theta} \theta\big)\big( \sum_{\theta \in \Theta_\tria} d_{\theta} \theta\big) = \beta_1 \sum_{\theta \in \Theta_\tria} c_\theta d_\theta \| h_\tria^s \theta\|_{L_2(\Omega)}^2.
\end{equation}
\end{lemma}

\begin{remark}
    It is not possible to construct $B_\tria \in \Lis(\W_\tria, \W_\tria')$ directly as a diagonal scaling operator.
    Indeed, this would require $\|w_\tria\|_\W \lesssim \|h_\tria^{s} w_\tria\|_{L_2(\Omega)}$ for $w_\tria \in \W_\tria$.
    Suppose this to be true, then by $L_2(\Omega)$-boundedness of the biorthogonal projector $P_\tria$,
    we would find for $v_\tria \in \V_\tria$ that
    \begin{align*}
        \|h_\tria^{-s}v_\tria\|_{L_2(\Omega)}
        &= \sup_{w \in L_2(\Omega)} \frac{\langle h_\tria^{-s} v_\tria, P_\tria w\rangle_{L_2(\Omega)}}{\| w\|_{L_2(\Omega)}}
        \lesssim \sup_{w \in L_2(\Omega)} \frac{\langle h_\tria^{-s} v_\tria, P_\tria  w \rangle_{L_2(\Omega)}}{\|P_\tria  w\|_{L_2(\Omega)}}\\
        &=  \sup_{w_\tria \in \W_\tria} \frac{\langle v_\tria, w_\tria \rangle_{L_2(\Omega)}}{\| h_\tria^s w_\tria\|_{L_2(\Omega)}}
        \lesssim \sup_{w_\tria \in \W_\tria} \frac{\langle v_\tria, w_\tria \rangle_{L_2(\Omega)}}{\| w_\tria\|_{\W}} \leq \|v_\tria\|_{\V},
    \end{align*}
    which is known not to be true for smooth functions in $\V_\tria$.
\end{remark}

Concluding: If, given a family of subspaces $\W_\tria \subset L_2(\Omega)$, one can find a family of projectors $Q_\tria \in \cL(L_2(\Omega),H^1_{0,\gamma}(\Omega)')$,
subspaces $\U_\tria \subseteq \ran Q_\tria$ (of finite element type) and $\B_\tria \subseteq  \ran (\identity-Q_\tria)$ such that
\begin{equation} \label{nesting}
\W_\tria \subset \ZZ_\tria:=\U_\tria+\B_\tria
\end{equation}
(with these spaces equipped with $\|\cdot\|_\W$-norm) and the conditions of Lemma~\ref{lem:boundedW} are satisfied, then given $B_\tria^\U \in \cLis(\U_\tria,\U_\tria')$ and $B_\tria^\B \in \cLis(\B_\tria,\B_\tria')$, the operator $B_\tria^\W$ defined by
\begin{equation} \label{defB}
(B_\tria w)(\tilde w):=(B_\tria^\U Q_\tria w)(Q_\tria \tilde w) + (B_\tria^\B( \identity - Q_\tria) w)((\identity - Q_\tria) \tilde w) \quad (w,\tilde w \in \W_\tria),
\end{equation}
is in $\cLis(\W_\tria,\W_\tria')$.
Moreover, assuming a uniformly $\|h_\tria^s \cdot\|_{L_2(\Omega)}$-stable basis for $\B_\tria$, the operator $B_\tria^\B$ can be of simple diagonal scaling type, where a natural definition for $B_\tria^\U$ is by $(B_\tria u)(\tilde u):=(B u)(\tilde u)$ ($u,\tilde u \in \U_\tria$) for some opposite order operator $B \in \cLis(\W,\W')$.
Finally, since $Q_\tria$ enters the implementation, we search this projector to be of local type.

\subsection{A space $\W_\tria$ decomposable into the piecewise constants and bubbles}\label{sec:decompose_pw_const}
In this subsection, we construct $\W_\tria=\Span \Psi_\tria$ such that both $\Psi_\tria$ is biorthogonal to $\Phi_\tria$
(Assumption~\eqref{eq:assumption_psi}),
and $\W_\tria$ allows an appropriate decomposition into the space of piecewise constants $\U_\tria:=\Ss_{\tria}^{-1,0}$ and a bubble space $\B_\tria$.

Fix $\tria \in \bbT$ and let $\tria_{*} \succ \tria$ be a uniform red-refinement, i.e.~every simplex
$T \in \tria$ is subdivided into $2^d$ subsimplices.\footnote{Red-refinement is not uniquely defined for $d \geq 3$, but the refined simplices at the corners of the `parent simplex'
are uniquely determined which suffices for our goal.}
We define $\Psi_\tria = \{ \psi_{\tria, \nu}\colon \nu \in N_\tria^0\} \subset  \Ss_{\tria_{*}}^{-1, 0}$ by taking a
weighted difference of `patch indicator' functions:
\begin{equation}\label{eq:biorth_patch}
    \psi_{\tria, \nu} := 2^{d+1} \1_{\omega_{\tria_{*}, \nu}} - \1_{\omega_{\tria, \nu}}  \quad (\nu \in N_\tria^0).
\end{equation}
\begin{lemma}\label{prop:existence_psi}
    The collection $\Psi_\tria$ satisfies Assumption~\eqref{eq:assumption_psi} with ${\supp \psi_{\tria,\nu} = \omega_{\tria,\nu}}$ and
\begin{equation}\label{eq:biorth_disc}
        \langle \psi_{\tria, \nu}, \phi_{\tria, \nu'} \rangle_{L_2(\Omega)} = \delta_{\nu \nu'} |\omega_{\tria,\nu}| \quad(\nu,\nu' \in N_\tria^0).
 \end{equation}
\end{lemma}
\begin{proof}
    Clearly $\supp \psi_{\tria, \nu} = \omega_{\tria, \nu}$, so we are left to show the biorthogonality condition.
    Fix some vertex $\nu \in N^0_\tria$. For a simplex $T_\nu \in \tria$ with $\nu \in T_\nu$, we have
    \[
        \langle \1_{T_\nu}, \phi_{\tria, \nu} \rangle_{L_2(\Omega)} = \frac{|T_{\nu}|}{d+1}.
     \]
     Let $T_{*,\nu} \in \tria_*$ be the (unique) simplex with $\nu \in T_{*,\nu} \subset T_\nu$.  From the refinement equation satisfied by the nodal hats, and $|T_{*,\nu}|=2^{-d}|T_\nu|$, it follows that
     \begin{align*}
         \langle \1_{T_{*,\nu}}, \phi_{\tria, \nu} \rangle_{L_2(\Omega)} &=
         \langle \1_{T_{*,\nu}}, \phi_{\tria_{*}, \nu} + \sum_{\nu \ne \tilde \nu \in N_{T_{*,\nu}}} 2^{-1} \phi_{\tria_{*}, \tilde \nu} \rangle_{L_2(\Omega)} = \frac{2^{-d} |T_\nu|}{d+1}(1 + 2^{-1}d),\\
         \langle \1_{T_{*,\nu}}, \phi_{\tria, \nu'} \rangle_{L_2(\Omega)} &= \cdots = \frac{2^{-d}|T_\nu|}{d+1} 2^{-1} \quad (\nu \ne \nu' \in N^0_{T_\nu}).
        \end{align*}
        From these relations~\eqref{eq:biorth_disc} follows.
\end{proof}

By Lemma~\ref{prop:existence_psi} it has been established that the Fortin interpolator is uniformly bounded, and that $D_\tria$ is represented by a diagonal matrix.
The next proposition verifies the conditions imposed in Sect.~\ref{stabledecomp} for the construction of $B_\tria$.

\begin{proposition}\label{prop:existence_biorth_bases}
Let $\U_\tria:= \Ss_{\tria}^{-1,0}$, $\W_\tria:=\Span \Psi_\tria$ as constructed above, $Q_\tria$ be the $L_2(\Omega)$-orthogonal projector onto~$\U_\tria$,
$\Theta_\tria:= (\identity - Q_\tria) \Psi_\tria$,
and
$\B_\tria:=\Span \Theta_\tria$.
Then $\W_\tria \subset \ZZ_\tria:=\U_\tria+\B_\tria$ {\rm(\eqref{nesting})}, the conditions of Lemma~\ref{lem:boundedW} are satisfied, in particular $Q_\tria \psi_\nu=\1_{\omega_\nu}$, and $\Theta_\tria$ is a uniformly $\|h_\tria^s \cdot\|_{L_2(\Omega)}$-stable basis for $\B_\tria$ as required for Lemma~\ref{lem:diagonal_b}.
\end{proposition}

\begin{proof}
The first statement follows from  $\W_\tria \subset L_2(\Omega)$.
The first two conditions of Lemma~\ref{lem:boundedW} are obviously valid.
Concerning the third condition, the inverse inequality $\|h_\tria \cdot\|_{L_2(\Omega)} \lesssim \| \cdot \|_{H_{0,\gamma}^1(\Omega)'}$ holds, for general $d$, on $\Ss_{\tria_{*}}^{-1, 0}$, see e.g.~\cite[Lem.~3.4]{249.97}, and thus in particular on $\ZZ_\tria$.
    The property $Q_\tria \psi_\nu=\1_{\omega_\nu}$ is easily checked.

We are left to show that the collection of bubbles $\{\theta_\nu:=(\identity-Q_\tria)\psi_\nu\colon \nu \in N_\tria^0\}$ is $\|h_\tria^s \cdot\|_{L_2(\Omega)}$-stable.
    Pick some $T \in \tria$, then the normalized `bubble element matrix' satisfies
    \begin{equation} \label{bubble-element}
    \begin{split}
      {\textstyle \frac{1}{4}} |T|^{-1} \langle \theta_\nu, \theta_{\nu'} \rangle_{L_2(T)} & = |T|^{-1}
        \langle 2^{d}\1_{\omega_{\tria_{*}, \nu}} - \1_{\omega_{\tria, \nu}},
        2^{d}\1_{\omega_{\tria_{*}, \nu'}} - \1_{\omega_{\tria, \nu'}} \rangle_{L_2(T)} \\
        & =
                \begin{cases}
            2^d - 1 & \nu = \nu'   \in N_T^0, \\
            - 1 & \nu \ne \nu'   \in N_T^0.
        \end{cases}
    \end{split}
    \end{equation}
    For $d \geq 2$, this constant (symmetric) $(d+1)\times(d+1)$ matrix is strictly diagonally dominant, and therefore positive definite. We conclude this proposition by
  \begin{align*}
        &\Big\|\sum_{\mathclap{\nu \in N_\tria^0}} h_\tria^s c_\nu \theta_\nu \Big\|^2_{L_2(\Omega)} = \sum_{T \in \tria} h_T^{2s} \Big\| \sum_{\mathclap{\nu \in N_T^0}} c_\nu \theta_\nu \Big\|^2_{L_2(T)}
       \eqsim  \sum_{T \in \tria} h_T^{2s} \sum_{\nu \in N_T^0} |c_\nu|^2 \|\theta_\nu\|_{L_2(T)}^2\\
        &= \sum_{\nu \in N_\tria^0} |c_\nu|^2 \sum_{T \in \tria} \|h_T^s \theta_\nu\|_{L_2(T)}^2=
        \sum_{\nu \in N_\tria^0} |c_\nu|^2 \|h_\tria^s \theta_\nu\|_{L_2(\Omega)}^2.\qedhere
   \end{align*}
\end{proof}
\begin{remark}
    For $d = 1$, the bubbles arising from $\Psi_\tria$ as given in~\eqref{eq:biorth_patch} do not form a $\|h_\tria^s\cdot\|_{L_2(\Omega)}$-stable collection.
    Instead, with $\tria_{**} \succ \tria$ being the two times uniform red-refinement, one can consider $\psi_{\tria, \nu} = \frac{16}{3} \1_{\omega_{\tria_{**}, \nu}} - \frac{1}{3} \1_{\omega_{\tria, \nu}}$ for which the statements of Lemma~\ref{prop:existence_psi} and
    Proposition~\ref{prop:existence_biorth_bases} are again valid.
\end{remark}

\subsubsection{Implementation}\label{sec:implementation_biorth_basis}
The matrix representation of preconditioner $\cF_{\Phi_\tria}^{-1} G_\tria (\cF'_{\Phi_\tria})^{-1}$ is given by
\[
    \bm{G}_\tria = \bm{D}^{-1}_\tria \bm{B}_\tria \bm{D}_\tria^{-\top}.
\]
With $\Psi_\tria$ as constructed in~\eqref{eq:biorth_patch}, we find that
$\bm{D}_\tria = \cF_{\Psi_\tria}' D_\tria \cF_{\Phi_\tria}$ is given by
\[
    \bm{D}_\tria = \diag \{ |\omega_\nu| \colon \nu \in N^0_\tria\}.
\]

Given some $B_\tria^\U \in \cLis(\U_\tria, \U_\tria')$
(recall that $\U_\tria = \Ss_\tria^{-1,0}$),
then by taking $B_\tria$ as described in~\eqref{defB}, we have
 \begin{align*}
\bm{B}_\tria &:= \cF'_{\Psi_\tria} B_\tria \cF_{\Psi_\tria}\\
             &= \cF'_{\Psi_\tria} (Q_\tria' B_\tria^\U Q_\tria + (\identity - Q_\tria)' B_\tria^\B (\identity - Q_\tria)) \cF_{\Psi_\tria}\\
             &=\bm{p}_\tria^\top \bm{B}_\tria^{\U} \bm{p}_\tria+ \bm{B}^{\B}_\tria,
\end{align*}
where, using that $\cF^{-1}_{\Theta_\tria}(\identity-Q_\tria)\cF_{\Psi_\tria} =\identity$ by $\Theta_\tria=(I-Q_\tria)\Psi_\tria$,
\[
\bm{B}_\tria^{\U}:=\cF_{\Sigma_\tria}' B_\tria^{\U}  \cF_{\Sigma_\tria}, \quad \bm{p}_\tria:=\cF^{-1}_{\Sigma_\tria}  Q_\tria\cF_{\Psi_\tria},\quad
\bm{B}^{\B}_\tria:=\cF_{\Theta_\tria}'B_\tria^{\B} \cF_{\Theta_\tria},
\]
Recall the canonical basis $\Sigma_\tria$ for $\U_\tria$ from~\eqref{sigma}.
Using $Q_\tria \psi_\nu = \1_{\omega_\nu}$ shows that
\[
(\bm{p}_\tria)_{T \nu} = \begin{cases} 1 & \text{if } T \subset \omega_\nu,\\
                                           0 & \text{else.}
                                        \end{cases}
\]
From~\eqref{bubble-element}, we infer that $\|h_\tria^s \theta_\nu\|_{L_2(\Omega)}^2 \eqsim |\omega_\nu|^{1+\frac{2s}{d}}$.
By making a harmless modification to the definition of $B_\tria^\B$ in~\eqref{eq:bubble_b} based on this equivalency, we obtain that
\[
\bm{B}_\tria^\B=\beta_1 \bm{D}_\tria^{1+\frac{2s}{d}}.
\]
The matrix $\bm{B}_\tria^{\U}$ depends on the operator $B_\tria^\U \in \cLis(\U_\tria, \U_\tria')$ that is selected.
The canonical choice is the Galerkin discretization operator on $\U_\tria$ of a $B \in \cLis(\W, \W')$.
The cost of the application of $\bm{G}_\tria$ is the cost of the application of $\bm{B}_\tria^{\U}$  plus cost that scales linearly in $\# \tria$.

\subsection{A space $\W_\tria$ decomposable into the continuous piecewise linears and bubbles}\label{sec:decompose_cont_pw_lin}
We follow the same program as in the previous subsection Sect.~\ref{sec:decompose_pw_const} but now with $\U_\tria:=\Ss_\tria^{0,1}$, being the space of continuous piecewise linears.

Other than in Sect.~\ref{sec:decompose_pw_const},
we cannot apply Proposition~\ref{proppie} for $Q_\tria$ being the orthogonal projector onto $\U_\tria$,
since with the current choice of this space it will not be a local projector.
As an alternative, we take $Q_\tria$ to be some biorthogonal projector. The question whether it enjoys an approximation property is answered in the following lemma.

\begin{lemma}\label{lemmie}
    For $\nu \in N_\tria$, so \emph{including} vertices on $\gamma$, let $\widetilde{\phi}_\nu \in L_2(\Omega)$ be such that
    \begin{equation} \label{properties}
    \| \widetilde{\phi}_\nu\|_{L_2(\Omega)} \lesssim h_\nu^{d/2}, \quad \sum\nolimits_{\nu \in N_\tria} \widetilde{\phi}_\nu = \1_{\Omega},
    \quad \supp \widetilde{\phi}_\nu \subset B(\nu;R h_\nu)
    \end{equation}
    for some constant $R>0$, and
      \[
        \big| \langle \widetilde{\phi}_\nu, \phi_{\nu'}\rangle_{L_2(\Omega)} \big|\eqsim \delta_{\nu \nu'} |\omega_\nu| \quad (\nu,\nu' \in N_\tria^0).
    \]
    Denote $\widetilde{\U}_\tria := \Span\{\widetilde{\phi}_\nu \colon \nu \in N_\tria^0\}$, so \emph{without} vertices on $\gamma$.
    The biorthogonal projector $Q_\tria \colon u \mapsto \sum_{\nu \in N_\tria^0} \frac{\langle u, \widetilde{\phi}_\nu\rangle_{L_2(\Omega)}}{\langle \phi_\nu, \widetilde{\phi}_\nu \rangle_{L_2(\Omega)}} \phi_\nu$, for which $\ran Q_\tria = \Ss_{\tria}^{0,1}$ and
    $\ran (\identity - Q_\tria) = \widetilde{\U}_\tria^{\perp_{L_2(\Omega)}}$, satisfies the approximation property
    \[
        \|h_\tria^{-1}(\identity - Q_\tria')v \|_{\cL(H_{0,\gamma}^1(\Omega),L_2(\Omega))} \lesssim 1,
    \]
    and $\|Q_\tria\|_{\cL(L_2(\Omega),L_2(\Omega))}\lesssim 1$.
\end{lemma}

\begin{proof}
      We use the same strategy as in~\cite{249.97}. That is, we define a Scott-Zhang type quasi-interpolator $\Pi_\tria \colon H^1(\Omega) \to L_2(\Omega)$, cf.~\cite{247.2}.
    For every $\nu \in N_\tria$, select a $(d-1)$-face $e_\nu$ of some $T \in \tria$ with $\nu \in e_\nu$ and $e_\nu \subset \gamma$ if $\nu \in \gamma$. We define $\Pi_\tria$ by
    \[
        \Pi_\tria u := \sum_{\nu \in N_\tria} g_{\tria, \nu}(u) \widetilde{\phi}_{\tria, \nu}, \quad g_{\tria,\nu}(u):=\fint_{e_\nu} u \, \dif s.
    \]
    Since $g_{\tria, \nu}(\1)=1$, using the properties from~\eqref{properties} one can show, cf.~proof of~\cite[Thm.~3.2]{249.97} for details, that
    \[
        \|h_\tria^{-1}(\identity - \Pi_\tria)(u)\| \lesssim \|u\|_{H^1(\Omega)} \quad (u \in H^1(\Omega)).
    \]
    By construction, $g_{\tria, \nu}(u) = 0$ for $\nu$ on $\gamma$ and $u \in H^1_{0, \gamma}(\Omega)$,
    and therefore $\ran \Pi_\tria|_{H_{0,\gamma}^1(\Omega)} \subset \widetilde{\U}_\tria$.
    Finally, combined with $L_2(\Omega)$-boundedness and locality of $Q_\tria'$, and the fact that $Q_\tria'$ reproduces $\widetilde{\U}_\tria$, we find that
    \begin{align*}
        \| h_\tria^{-1}(\identity - Q_\tria')v\|_{L_2(\Omega)} &= \inf_{w_\tria \in \widetilde{\U}_\tria} \| h_\tria^{-1}(\identity - Q_\tria')(v - w_\tria) \|_{L_2(\Omega)} \\
                                                               &\lesssim \|h_\tria^{-1}(\identity - \Pi_\tria)(v) \|_{L_2(\Omega)} \lesssim \| v\|_{H^1_{0,\gamma}(\Omega)} \quad (v \in H^1_{0,\gamma}(\Omega)).
    \end{align*}
    The last statement can be proven similarly as in the proof of Theorem~\ref{thm:boundedbiorth}.
\end{proof}

As before, let $\tria_* \succ \tria$ denote a uniform red-refinement of $\tria$, and
for any $T \in \tria$ and $\nu \in N_T$, let $T_{*,\nu} \in \tria_*$ denote the simplex with $\nu \in T_{*,\nu} \subset T$.
For $\nu \in N_\tria$, so including boundary vertices, define
    \[
        \widetilde{\phi}_{\tria, \nu} := {\textstyle \frac{1}{d+1}} \sum_{\substack{T \in \tria \\ T \subset \omega_\nu}}\big( \1_T + {\textstyle \frac{d 2^{1+d}}{d+1}}  \1_{T_{*,\nu}} - {\textstyle \frac{2^{1+d}}{d+1}} \sum_{\substack{\nu' \in N_T\\ \nu' \ne \nu}} \1_{T_{*,\nu'}}\big) \in \Ss_{\tria_*}^{-1,0}.
    \]
    These functions satisfy~\eqref{properties}, and
    \[
        \langle \widetilde{\phi}_{\tria, \nu}, \phi_{\tria, \nu'} \rangle_{L_2(\Omega)} = \delta_{\nu \nu'} (d+1)^{-1} |\omega_{\tria, \nu}|,
           \]
           and so determine a valid biorthogonal projector $Q_\tria$ via Lemma~\ref{lemmie}.

    For $\tria_{**} \succ \tria_*$ a uniform red-refinement of~$\tria_*$, we define $\Theta_\tria:=\{ \theta_{\tria,\nu}\colon \nu \in N_\tria^0\}$ by
    \[
        \theta_{\tria, \nu} := {\textstyle \frac{2^{d+2}}{d+2}} \big(2^{d}\1_{\omega_{\tria_{**}, \nu}} -  \1_{\omega_{\tria_*, \nu}}\big).
    \]
    Since red-refinement subdivides each simplex into $d$ subsimplices, one infers that
    \begin{equation} \label{bubble-orthog}
    \B_\tria:=\Span \Theta_\tria \perp_{L_2(\Omega)} \Ss_{\tria_*}^{-1,0},
    \end{equation}
    so that in particular $\B_\tria \subset \ker Q_\tria$.

Defining $\Psi_\tria := \{\psi_{\tria, \nu}\colon \nu \in N_\tria^0\}$ by
 \[
 \psi_{\tria, \nu} := \phi_{\tria, \nu} +\theta_{\tria, \nu},
 \]
calculations as in the proof of Lemma~\ref{prop:existence_psi} show the following result.
 \begin{lemma}
    The collection $\Psi_\tria$ satisfies Assumption~\eqref{eq:assumption_psi} with ${\supp \psi_{\tria,\nu} = \omega_{\tria,\nu}}$ and
    \begin{equation}\label{eq:biorth_cons}
            \langle \psi_{\tria, \nu}, \phi_{\tria, \nu'} \rangle_{L_2(\Omega)} = \delta_{\nu \nu'} (d+1)^{-1} |\omega_{\tria,\nu}| \quad(\nu,\nu' \in N_\tria^0).
     \end{equation}
\end{lemma}

So the Fortin interpolator is uniformly bounded, and $D_\tria$ is represented by a diagonal matrix. Next we verify
the conditions imposed in Sect.~\ref{stabledecomp} for the construction of $B_\tria$.

\begin{proposition}\label{prop:existence_biorth_costruction_cont}
     Let $\U_\tria$, $Q_\tria$, $\B_\tria$, and $\W_\tria := \Span \Psi_\tria$ be defined as above.
Then $\W_\tria \subset \ZZ_\tria:=\U_\tria+\B_\tria$ {\rm(\eqref{nesting})}, the conditions of Lemma~\ref{lem:boundedW} are satisfied,
in particular $\Phi_\tria = Q_\tria \Psi_\tria$ and so $\Theta_\tria =  (\identity - Q_\tria) \Psi_\tria$,
and lastly, $\Theta_\tria$ is an $\|h_\tria^s \cdot\|_{L_2(\Omega)}$-orthogonal basis for $\B_\tria$ as required for Lemma~\ref{lem:diagonal_b}.
\end{proposition}

\begin{proof} The first statement is obviously true. We have already verified the first two conditions of Lemma~\ref{lem:boundedW}. The third
    condition follows from this inverse inequality on $\Ss_{\tria_{**}}^{-1,1}$ (see e.g.~\cite[(5.14)]{249.97}), and
   $\Phi_\tria = Q_\tria \Psi_\tria$ is a consequence of~\eqref{bubble-orthog}.
The last statement follows from $|\supp \theta_{\nu}\cap \supp \theta_{\nu'}|=0$ when $\nu \neq \nu'$.
\end{proof}

\subsubsection{Implementation}\label{sec:implementation_biorth_basis_cont}
Suppose that we have some operator $B_\tria^\U \in \cLis(\U_\tria, \U_\tria')$ uniformly (here $\U_\tria = \Ss_\tria^{0,1}$).
The matrix representation of the preconditioner $G_\tria$, with
$B_\tria$ from~\eqref{defB} and
the bases from Proposition~\ref{prop:existence_biorth_costruction_cont}, becomes
%and the operator $B_\tria^{\B}$ from Lemma~\ref{lem:diagonal_b}, becomes
\begin{align*}
\bm{G}_\tria &= \bm{D}^{-1}_\tria \bm{B}_\tria \bm{D}_\tria^{-\top},\\
\bm{B}_\tria &:= \cF'_{\Psi_\tria} (Q_\tria' B_\tria^\U Q_\tria + (\identity - Q_\tria)' B_\tria^\B (\identity - Q_\tria)) \cF_{\Psi_\tria}\\
& =\bm{B}_\tria^{\U} + \bm{B}^{\B}_\tria,
\end{align*}
 with these matrices given by
\begin{align*}
    \bm{D}_\tria &:= \cF_{\Psi_\tria}' D_\tria \cF_{\Phi_\tria} = \diag \big\{ {\textstyle \frac{|\omega_\nu|}{d+1}} \colon \nu \in N^0_\tria\big\},\\
    \bm{B}_\tria^{\U}&:=\cF_{\Phi_\tria}' B_\tria^{\U}  \cF_{\Phi_\tria}, \quad \bm{B}^{\B}_\tria:= \cF_{\Theta_\tria}'B_\tria^{\B} \cF_{\Theta_\tria}= \beta_1 \bm{D}_\tria^{1+ \frac{2s}{d}},
\end{align*}
where we used that $\cF^{-1}_{\Phi_\tria}  Q_\tria\cF_{\Psi_\tria} = \bm{\identity}$ and $\cF^{-1}_{\Theta_\tria}(\identity-Q_\tria)\cF_{\Psi_\tria} = \bm{\identity}$, and
where, based on $\|h_\tria^s \theta_\nu\|_{L_2(\Omega)}^2 \eqsim |\omega_\nu|^{1+\frac{2s}{d}}$,
we made an harmless modification to the operator $B_\tria^{\B}$ from Lemma~\ref{lem:diagonal_b}.

\section{Extensions}\label{sec:extensions}

\subsection{Higher order}
Add the superscript $1$ to the spaces defined so far,
e.g.~write $\V_\tria^1$ for $\Ss_\tria^{0,1}$ with its nodal basis $\Phi_\tria^1$, and similarly use $G_\tria^1$ for the associated preconditioner from either Sect.~\ref{sec:decompose_pw_const} or Sect.~\ref{sec:decompose_cont_pw_lin}.

We  will now consider a (family of) higher order continuous piecewise polynomials, i.e.~for some $\ell \in \{2, 3, \dots \}$ let
\[
    \V_\tria = \Ss_\tria^{0, \ell} := \{ u \in H_{0,\gamma}^1(\Omega) \colon u|_T \in \cP_\ell \,(T \in \tria)\} \subset \V.
\]
Because we have an inverse inequality on $\V_\tria$, we can construct a uniform preconditioner $G_\tria \in \Lis(\V_\tria', \V_\tria)$ using an additive subspace correction method.
That is, we consider
the overlapping decomposition $\V_\tria = \V_\tria^1 + \V_\tria^2$, where
these spaces are given by
\[
    \V_\tria =  (\V_\tria, \| \cdot \|_\V), \quad \V_\tria^1 = (\V_\tria^1, \|\cdot \|_\V), \quad \V_\tria^2 = (\V_\tria, \|h_\tria^{-s} \cdot \|_{L_2(\Omega)}).
\]
\begin{proposition}
For $k\in \{1,2\}$, let $G_\tria^k \in \cLis((\V_\tria^k)', \V_\tria^k)$, then for $I_\tria^k : \V_\tria^k \to \V_\tria$ the trivial embedding, we find that
$G_\tria := \sum_{k = 1}^2 I_\tria^k G_\tria^k (I_\tria^k)' \in \cLis(\V_\tria', \V_\tria)$,
with
\begin{align*}
    \|G_\tria\|_{\cL(\V_\tria', \V_\tria)} &\lesssim \max_{k = 1,2} \|G_\tria^k\|_{\cL((\V_\tria^k)', \V_\tria^k)},\\
    \|\Re(G_\tria)^{-1}\|_{\cL(\V_\tria,\V_\tria')} &\lesssim \max_{k=1,2} \| \Re(G_\tria^k)^{-1}\|_{\cL(\V_\tria^k, (\V_\tria^k)')}.
\end{align*}
\end{proposition}
\begin{proof}
    We have the (standard) inverse inequality $\| u\|_\V \lesssim \|h_\tria^{-s} u\|_{L_2(\Omega)}$ for $u \in \V_\tria$.
    Let $u \in \V_\tria$, then for any $(u_1, u_2) \in \V_\tria^1 \times \V_\tria^2$ with $u_1 + u_2 = u$ we find
    \[
    \|u \|_\V \leq \|u_1\|_\V +\|u_2\|_\V \lesssim \|u_1\|_\V + \|h_\tria^{-s}u_2\|_{L_2(\Omega)}.
    \]
    Denote $\Pi_\tria^1 \colon H_{0, \gamma}^1(\Omega) \to \V_\tria^1$ for the Scott-Zhang interpolator (\cite{247.2}).
    For $u \in \V_\tria$, take $u_1 = \Pi_\tria^1 u \in \V_\tria^1$ and $u_2 = u - \Pi_\tria^1 u \in \V_\tria^2$, then from approximation properties of the interpolator we infer
    \begin{align*}
     \|u_1\|_\V + \|h_\tria^{-s} u_2\|_{L_2(\Omega)} &\leq \|u\|_\V +\|u_2\|_\V+ \|h_\tria^{-s} u_2\|_{L_2(\Omega)}\\
     &\lesssim \|u\|_\V + \|h_\tria^{-s} u_2\|_{L_2(\Omega)} \lesssim \|u\|_\V.
    \end{align*}
    Since apparently for $u \in \V_\tria$,
    \[
    \|u\|_\V\eqsim \inf\big\{\|u_1\|_\V + \|h_\tria^{-s}u_2\|_{L_2(\Omega)}\colon u_1 \in \V_1,\,u_2\in \V_2,\,u_1+u_2=u\big\},
    \]
     the result follows from theory of subspace correction methods, e.g.~\cite{242.5}.
\end{proof}
On the space $\V_\tria^1$ we can apply the operator $G_\tria^1$ constructed earlier, whereas on $\V_\tria^2$ a simple scaling operator suffices.
Denote $N^{0, \ell}_\tria$ for the set of canonical Lagrange evaluation points of $\Ss_\tria^{0,\ell}$, and let $\Phi^\ell_\tria = \{ \phi^\ell_{\trianu} \colon \nu \in N^{0,\ell}_\tria \}$ be the corresponding nodal basis.
For some constant $\beta_2>0$, define an operator $R_\tria\colon \V_\tria^2 \to (\V_\tria^2)'$ by
\[
    (R_\tria u)(w) := \beta_2^{-1} \sum_{\nu \in N^{0,\ell}_\tria} \|h_\tria^{-s} \phi^\ell_\nu\|_{L_2(\Omega)}^2 u(\nu)w(\nu).
\]
\begin{proposition}\label{prop:scaling_higher_order}
    The operator $G^2_\tria := R_\tria^{-1}$ satisfies $G^2_\tria  \in \cLis((\V_\tria^2)', \V_\tria^2)$ uniformly.
\end{proposition}
\begin{proof}
    It is not hard to see that the result follows if $\Phi^\ell_\tria$ is a (uniformly) $\|h_\tria^{-s} \cdot \|_{L_2(\Omega)}$-stable basis. Writing $N_T^{0, \ell} := N_\tria^{0,\ell} \cap T$, this stability can be deduced from
        \begin{align*}
            \Big\|h_\tria^{-s} \sum_{\mathclap{\nu \in N^{0, \ell}_\tria}} c_\nu \phi_\nu^\ell \Big\|^2_{L_2(\Omega)} &= \sum_{T \in \tria} h_T^{-2s} \Big\| \sum_{\mathclap{\nu \in N_T^{0,\ell}}} c_\nu \phi_\nu^\ell \Big\|^2_{L_2(T)} \\
                &\eqsim \sum_{T \in \tria} h_T^{-2s} \sum_{\nu \in N_T^{0,\ell}} |c_\nu|^2 \| \phi_\nu^\ell\|^2_{L_2(T)} = \sum_{\nu \in N_\tria^{0,\ell}} |c_\nu|^2 \|h_\tria^{-s} \phi_\nu^\ell\|^2_{L_2(\Omega)}.\qedhere
        \end{align*}
\end{proof}

\subsubsection{Implementation}\label{sec:implementation_higher_order}
Equipping $\V_\tria$ and $\V_\tria^2$ by $\Phi_\tria^\ell$, and $\V_\tria^1$ by $\Phi_\tria^1$, the matrix representation of $G_\tria := \sum_{k = 1}^2 I_\tria^k G_\tria^k (I_\tria^k)' \in \cLis(\V_\tria', \V_\tria)$   is given by
\[
\bm{G}_\tria=\bm{q}_\tria \bm{G}_\tria^1 \bm{q}_\tria^\top +\bm{G}_\tria^{2},
\]
with $\bm{G}_\tria^1 $ either from Sect.~\ref{sec:implementation_biorth_basis} or Sect.~\ref{sec:implementation_biorth_basis_cont},
\[
    (\bm{q}_\tria)_{\nu' \nu} = \phi^\ell_{\trianu'}(\nu) \quad (\nu' \in N^{0,\ell}_\tria, \nu \in N^{0,1}_\tria).
\]
and
\[
    \bm{G}_\tria^2 = \beta_2 \diag \{  \|h_\tria^{-s}\phi_\nu^\ell\|_{L_2(\Omega)}^{-2}\colon \nu \in N^{0,\ell}_\tria\}.
\]

\subsection{Manifolds}
Let $\Gamma$ be a compact
  $d$-dimensional Lipschitz, piecewise smooth manifold in $\R^{d'}$ for some $d' \geq d$ with or without boundary $\partial\Gamma$.
 For some closed measurable $\gamma \subset \partial\Gamma$ and $s \in [0,1]$, let
 \[
\V:=[L_2(\Gamma),H^1_{0,\gamma}(\Gamma)]_{s,2},\quad \W:=\V'.
\]
 We assume that $\Gamma$ is given as the closure of the disjoint union of $\cup_{i=1}^p \chi_i(\Omega_i)$, with, for $1 \leq i \leq p$,
$\chi_i\colon \R^d \rightarrow \R^{d'}$ being some smooth regular parametrization, and $\Omega_i \subset \R^d$ an open polytope.
 W.l.o.g.\ assuming that for $i \neq j$, $\overline{\Omega}_i \cap \overline{\Omega}_j=\emptyset$, we define
 \[
 \chi\colon \Omega:=\cup_{i=1}^p \Omega_i \rightarrow \cup_{i=1}^p \chi_i(\Omega_i) \text{ by } \chi|_{\Omega_i}=\chi_i.
 \]

Let $\bbT$ be a family of conforming partitions $\tria$ of $\Gamma$ into `panels' such that, for $1 \leq i \leq p$,  $\chi^{-1}(\tria) \cap \Omega_i$ is a uniformly shape regular conforming partition of $\Omega_i$ into $d$-simplices (that for $d=1$ satisfies a uniform $K$-mesh property).
We assume that $\gamma$ is a (possibly empty) union of `faces' of $T \in \tria$ (i.e., sets of type $\chi_i(e)$, where $e$ is a $(d-1)$-dimensional face of $\chi_i^{-1}(T)$).

The usual lowest order boundary element spaces are defined by
\begin{align*}
\Ss^{-1,0}_\tria&:=\{u \in L_2(\Gamma)\colon u \circ \chi |_{\chi^{-1}(T)} \in \cP_0 \,\,(T \in \tria)\},,\\
\Ss^{0,1}_{\tria}&:=\{u \in H^1_{0,\gamma}(\Gamma)\colon u \circ \chi |_{\chi^{-1}(T)} \in \cP_1 \,\,(T \in \tria)\},
\end{align*}
with their canonical bases denoted as $\Sigma_\tria=\{\1_T\colon T \in \tria\}$
and $\Phi_\tria=\{\phi_\nu\colon \nu \in N_\tria^0\}$, respectively, with $N_\tria^0$ the vertices of $\tria$ not on $\gamma$.

The construction of the preconditioners in the domain case relied on the explicit construction of a collection $\Psi_\tria$ biorthogonal to $\Phi_\tria$, and on the explicit computation of a (bi)orthogonal projection of $\W_\tria:=\Span \Psi_\tria$ onto either $\Ss^{-1,0}_\tria$ or $\Ss^{0,1}_{\tria}$, where orthogonality was interpreted w.r.t.~the $L_2(\Omega)$-scalar product.
Both the construction of $\Psi_\tria$ and the computation of the (bi)orthogonal projection could be reduced to computations on the individual elements in the partition, which yielded explicit expressions.

When attempting to transfer everything to the manifold case, a problem is the appearance of a generally non-constant weight $x \in |\partial\chi(x)|$ in $L_2(\Gamma)$-scalar product
\[
\langle u,v\rangle_{L_2(\Gamma)}=\int_\Omega u(\chi(x)) v(\chi(x))|\partial\chi(x)|\,dx.
\]
To deal with this, following \cite{249.97}, on $L_2(\Gamma)$ we define an additional `mesh-dependent' scalar product
\[
 \langle u,v\rangle_{\tria}:=\sum_{T \in \tria} \frac{|T|}{|\chi^{-1}(T)|} \int_{\chi^{-1}(T)} u(\chi(x))v(\chi(x))dx.
\]
which is constructed by replacing on each $\chi^{-1}(T)$, the Jacobian $|\partial \chi|$ by its average $\frac{|T|}{|\chi^{-1}(T)|}$ over $\chi^{-1}(T)$, and interpret (bi)orthogonality with respect to this scalar product.

Now all steps in the \emph{construction} of the preconditioners carry over, and yield preconditioners for the manifold case whose implementations are exactly as described in Sect.~\ref{sec:implementation_biorth_basis} and Sect.~\ref{sec:implementation_biorth_basis_cont},
where the patch volumes $|\omega_{\tria,\nu}|$ now should be read as the volumes of the patches on $\Gamma$.

To \emph{prove} that the constructed preconditioners are indeed uniform preconditioners requires additional work due to the use of the mesh-dependent scalar product. We refer to \cite{249.97} for details.
The key ingredient is that not only the norm associated to $\langle \cdot,\cdot\rangle_{L_2(\Gamma)}$ is uniformly equivalent to $\|\cdot\|_{L_2(\Gamma)}$, but also that  $\langle\cdot,\cdot\rangle_{L_2(\Gamma)}$ and $ \langle \cdot,\cdot \rangle_{\tria}$ are close in the sense that
\begin{equation}\label{closeness}
|\langle v,u\rangle_\tria-\langle v,u\rangle_{L_2(\Gamma)}| \lesssim \|h_\tria v\|_{L_2(\Gamma)} \|u\|_{L_2(\Gamma)} \quad (v,u \in L_2(\Gamma)).
\end{equation}

\section{Numerical Results}\label{sec:numerical_results}
Let $\Gamma = \partial [0,1]^3 \subset \R^3$ be the boundary of the unit cube, $\V := H^{1/2}(\Gamma)$, $\W := H^{-1/2}(\Gamma)$,
and ${\V_\tria=\Ss^{0,\ell}_\tria} \subset \V$ the trial space of continuous piecewise polynomials of degree $\ell$ w.r.t.~a partition~$\tria$.
We shall evaluate preconditioning of essentially a discretized Hypersingular Integral operator.

The Hypersingular Integral operator $\tilde A \in \cL(\V, \V')$ is only semi-coercive, since it has a non-trivial kernel equal to $\Span\{\1\}$.
Solving $\tilde Au=f$ for $f$ with $f(\1)=0$ is, however, equivalent to solving $A u =f$ with $A$ given by $(Au)(v) = (\tilde Au)(v) + \alpha \langle u, \1 \rangle_{L_2(\Gamma)} \langle v, \1 \rangle_{L_2(\Gamma)}$ for some $\alpha > 0$.
This operator $A $ is in $\cLis(\V, \V')$, and we shall consider preconditioning discretizations $A_\tria \in \cLis(\V_\tria, \V_\tria')$ of $A$.
By comparing different values numerically, we found $\alpha = 0.05$ to give good results in our examples.

As opposite order operator $B$ we take the Weakly Singular integral operator, which on compact $2$-dimensional manifolds is known to be in $\cLis(\W, \W')$.
We will compare preconditioners $G_\tria$ based on the discretizations $B_\tria^\U \in \cLis(\U_\tria, \U_\tria')$ of $B$, for $\U_\tria = \Ss_{\tria}^{-1,0}$ or $\U_\tria = \Ss_{\tria}^{0,1}$ equipped with the canonical bases $\Sigma_\tria=\{\1_T\colon T \in \tria\}$ and $\Phi_\tria=\{\phi_\nu\colon \nu \in N_\tria\}$, respectively,
cf.~Sect.~\ref{sec:implementation_biorth_basis} or Sect.~\ref{sec:implementation_biorth_basis_cont}.

For $\ell = 1$ (the lowest order case) and $\V_\tria$ being equipped with $\Phi_\tria$, the matrix representation of the preconditioner $G_\tria$ reads either as (Sect.~\ref{sec:implementation_biorth_basis})
\[
\bm{G}_\tria=\bm{G}^{-1,0}_\tria=\bm{D}_\tria^{-1}\big(\bm{p}_\tria^\top \bm{B}_\tria^\U \bm{p}_\tria+\beta_1 \bm{D}_\tria^{3/2}\big)\bm{D}_\tria^{-1}
\]
where $\bm{B}_\tria^\U=(B \Sigma_\tria)(\Sigma_\tria)$, $\bm{D}_\tria = \diag \{ |\omega_\nu| \colon \nu \in N_\tria\}$,
$
(\bm{p}_\tria)_{T \nu} = \begin{cases} 1 & \text{if } T \subset \omega_\nu,\\
                                           0 & \text{otherwise,}
                                        \end{cases}
$ and $\beta_1>0$ is some constant, or
as (Sect.~\ref{sec:implementation_biorth_basis_cont})
\[
\bm{G}_\tria=\bm{G}^{0,1}_\tria=\bm{D}_\tria^{-1}\big(\bm{B}_\tria^\U+\beta_1 \bm{D}_\tria^{3/2}\big)\bm{D}_\tria^{-1}
\]
where $\bm{B}_\tria^\U=(B \Phi_\tria)(\Phi_\tria)$, $\bm{D}_\tria = \diag \{ |\frac{\omega_\nu}{d+1}| \colon \nu \in N_\tria\}$, and $\beta_1>0$ is some constant.

For $\ell > 1$ denote the above $\bm{G}_\tria$ by either $\bm{G}^{1,-1,0}_\tria$ or $\bm{G}^{1,0,1}_\tria$, then, with $\V_\tria=\Ss_\tria^{0,\ell}$ being equipped with the standard nodal basis $\{\phi^\ell_\nu\colon \nu \in N_\tria^{\ell}\}$, the matrix representation of the preconditioner $G_\tria \in \cLis((\Ss_\tria^{0,\ell})', \Ss_\tria^{0,\ell})$ from Sect.~\ref{sec:implementation_higher_order} is
\[
    \bm{G}^{*}_\tria = \bm{q}_\tria \bm{G}_\tria^{1,*} \bm{q}_\tria^\top +  \beta_2 \diag \{  \|h_\tria^{-\frac12}\phi_\nu^\ell\|_{L_2(\Omega)}^{-2}\colon \nu \in N^{\ell}_\tria\},
\]
where either $*=-1,0$ or $*=0,1$, and
$(\bm{q}_\tria)_{\nu' \nu} = \phi^\ell_{\trianu'}(\nu)$ ($\nu' \in N^{\ell}_\tria, \nu \in N^{1}_\tria$).

The (full) matrix representations of the discretized singular integral operators $\bm{A}_\tria$ and $\bm{B}_\tria^{\U}$ are calculated using the BEM++ software package~\cite{249.04}.
Condition numbers are determined using Lanczos iteration with respect to $\lnrm\cdot\rnrm:=\|\bm{A}_\tria^{\frac{1}{2}}\cdot\|$.

\subsection{Uniform refinements}
Consider a conforming triangulation $\tria_1$ of $\Gamma$ consisting of $2$ triangles per side, so $12$ triangles with $8$ vertices in total.
We let $\bbT$ be the sequence $\{\tria_k\}_{k \geq 1}$ of uniform newest vertex bisections, where $\tria_k \succ \tria_{k-1}$ is found by bisecting each triangle from $\tria_{k-1}$.

With $\V_\tria=\Ss_\tria^{0,1}$, Table~\ref{tbl:unif} compares the condition numbers for the preconditioned system given by Sect.~\ref{sec:implementation_biorth_basis} ($\U_\tria = \Ss_\tria^{-1,0}$)
and by Sect.~\ref{sec:implementation_biorth_basis_cont} ($\U_\tria = \Ss_{\tria}^{0,1}$).
We see that the condition numbers remain nicely bounded, and that both choices give similar condition numbers.

Instead of using the `full matrices', we can consider compressed hierarchical matrices to approximate the stiffness matrices $\bf{A}_\tria$ and $\bf{B}^\U_\tria$ for finer partitions.
Table~\ref{tbl:hmat} gives the condition numbers, again for uniform refinements, but now using hierarchical matrices based on adaptive cross approximation~\cite{127.7, 19.896}.
We see that even for large systems, our preconditioner gives very satisfactory results.

Finally, consider the (higher order) trial space $\V_\tria = \Ss_\tria^{0,3}$. Table~\ref{tbl:higherorder} gives condition numbers for the preconditioned system, using the method described in
Sect.~\ref{sec:implementation_higher_order}.

\begin{table}[]
\centering
\caption{Spectral condition numbers of the preconditioned hypersingular system, using uniform refinements, discretized by continuous piecewise linears $\Ss_\tria^{0,1}$, with $\alpha = 0.05$.
    The preconditioners $\bm{G}_\tria^{-1,0}$ and $\bm{G}_\tria^{0,1}$ are constructed using the single layer operator discretized on $\U_\tria =\Ss_\tria^{-1,0}$ (Sect.~\ref{sec:implementation_biorth_basis}) and $\U_\tria = \Ss_\tria^{0,1}$ (Sec~\ref{sec:implementation_biorth_basis_cont}), respectively, where have used $\beta_1 = 0.65$ in the first case and $\beta_1 = 0.34$ in the second case.
}
\begin{tabular}{rrrr}
\toprule
dofs &  $\kappa_S(\bm{A}_\tria)$ & $\kappa_S(\bm{G}_\tria^{-1,0} \bm{A}_\tria)$ & $\kappa_S(\bm{G}_\tria^{0,1} \bm{A}_\tria)$ \\
\midrule
$   14$&$   3.0$&$ 2.71$&$2.64$\\
$   50$&$   7.1$&$ 2.36$&$2.37$\\
$  194$&$  14.2$&$ 2.25$&$2.26$\\
$  770$&$  28.7$&$ 2.30$&$2.27$\\
$ 3074$&$  57.8$&$ 2.29$&$2.27$\\
$12290$&$ 115.7$&$ 2.29$&$2.27$\\
$49154$&$ 231.4$&$ 2.30$&$2.27$\\
\bottomrule
\end{tabular}\label{tbl:unif}
\end{table}

\begin{table}[]
\centering
\caption{In the same setting as Table~\ref{tbl:unif}, but using compressed hierarchical matrices.}
\begin{tabular}{rrrr}
\toprule
dofs &  $\kappa_S(\bm{A}_\tria)$ & $\kappa_S(\bm{G}_\tria^{-1,0} \bm{A}_\tria)$ & $\kappa_S(\bm{G}_\tria^{0,1} \bm{A}_\tria)$ \\
\midrule
$ 12290$&$ 115.6$&$ 2.29$& $2.27$ \\
$ 24578$&$ 168.7$&$ 2.24$& $2.24$ \\
$ 49154$&$ 231.3$&$ 2.30$& $2.27$ \\
$ 98306$&$ 336.9$&$ 2.25$& $2.25$ \\
$196610$&$ 461.7$&$ 2.30$& $2.28$ \\
$393218$&$ 671.9$&$ 2.27$& $2.28$ \\
$786434$&$ 751.6$&$ 2.30$& $2.30$ \\
\bottomrule
\end{tabular}\label{tbl:hmat}
\end{table}

\begin{table}[]
\centering
\caption{Spectral condition numbers of the preconditioned hypersingular system, using uniform refinements, discretized by continuous piecewise cubics $\Ss_\tria^{0,3}$, with $\alpha = 0.05$.
    The higher order preconditioners $\bm{G}_\tria^{-1,0}$ and $\bm{G}_\tria^{0,1}$ are constructed as described in Sect.~\ref{sec:implementation_higher_order}, by using the
    preconditioners from Table~\ref{tbl:unif} with
constants $\beta_1 = 0.65, \beta_2 =0.065$ in the first case and $\beta_1 = 0.34, \beta_2 = 0.065$ in the second case.}
\begin{tabular}{rrrr}
\toprule
dofs &  $\kappa_S(\bm{A}_\tria)$ & $\kappa_S(\bm{G}_\tria^{-1,0} \bm{A}_\tria)$ & $\kappa_S(\bm{G}_\tria^{0,1} \bm{A}_\tria)$ \\
\midrule
$   56 $&$   19.49 $&$     4.75  $&$     4.72$\\
$  218 $&$   36.27 $&$     5.18  $&$     5.17$\\
$  866 $&$   74.78 $&$     6.23  $&$     6.20$\\
$ 3458 $&$  150.73 $&$     6.55  $&$     6.48$\\
$13826 $&$  301.97 $&$     6.63  $&$     6.57$\\
$55298 $&$  603.86 $&$     6.65  $&$     6.58$\\
\bottomrule
\end{tabular}\label{tbl:higherorder}
\end{table}

\subsection{Local refinements}
Here we take $\bbT$ to be the sequence $\{\tria_k\}_{k \geq 1}$ of locally refined triangulations, where $\tria_k \succ \tria_{k-1}$ is constructed using
conforming newest vertex bisection to refine all triangles in $\tria_{k-1}$ that touch a corner of the cube.

Table~\ref{tbl:local} gives condition numbers of the preconditioned hypersingular system discretized by continuous piecewise linears, i.e.~$\V_\tria = \Ss_\tria^{0,1}$.
The condition numbers remain bounded under local refinements, confirming uniformity of the preconditioner w.r.t.~$\bbT$.

\begin{table}[]
\centering
\caption{Spectral condition numbers of the preconditioned hypersingular system discretized by $\Ss_\tria^{0,1}$ using local refinements at each of the eight cube corners.
    Both preconditioners $\bm{G}_\tria^{-1,0}$ and $\bm{G}_\tria^{0,1}$ are constructed with same parameters as in Table~\ref{tbl:unif}, and are compared against
diagonal preconditioning.  The second column is defined by $h_{\tria, min} := \min_{T \in \tria} h_T$.}
\begin{tabular}{@{}rlccc@{}}
\toprule
dofs     &  $h_{\tria,min}$   &$\kappa_S(\diag(\bm{A}_\tria)^{-1}\bm{A}_\tria) $ & $\kappa_S(\bm{G}_\tria^{-1,0} \bm{A}_\tria)$ & $\kappa_S(\bm{G}_\tria^{0,1} \bm{A}_\tria)$ \\
\midrule
$    8 $&$  1.4\cdot 10^0     $&$      2.15 $&$      2.83 $&$      2.68$\\
$   14 $&$  1.0\cdot 10^0     $&$      2.79 $&$      2.71 $&$      2.64$\\
$  314 $&$  1.1\cdot 10^{-2}  $&$     12.11 $&$      2.21 $&$      2.20$\\
$  626 $&$  1.2\cdot 10^{-4}  $&$     13.18 $&$      2.31 $&$      2.30$\\
$  938 $&$  1.3\cdot 10^{-6}  $&$     13.43 $&$      2.36 $&$      2.36$\\
$ 1250 $&$  1.4\cdot 10^{-8}  $&$     13.51 $&$      2.39 $&$      2.38$\\
$ 1562 $&$  1.6\cdot 10^{-10} $&$     13.53 $&$      2.41 $&$      2.39$\\
$ 1850 $&$  2.5\cdot 10^{-12} $&$     13.55 $&$      2.41 $&$      2.40$\\
\bottomrule
\end{tabular}\label{tbl:local}
\end{table}

\section{Conclusion}
Using the framework of operator preconditioning, we have constructed uniform preconditioners for elliptic operators of orders $2s \in [0,2]$ discretized by continuous finite (or boundary) elements.
The evaluation of the preconditioners requires the application of an opposite order operator plus minor cost of linear complexity. Compared to earlier proposals, both the construction of a so-called dual-mesh and the inversion of a non-diagonal matrix are avoided, and our results are valid without constraints on the mesh-grading.
For lowest order finite elements the computed condition numbers of the preconditioned system are below $2.5$.
\bibliographystyle{alpha}
%\bibliography{../ref.bib}
%\bibliography{../../ref.bib}
\newcommand{\etalchar}[1]{$^{#1}$}

\end{document}